\documentclass[11pt]{amsart}

\usepackage{pdfpages}
\usepackage{amssymb}
\usepackage{amsmath}
\usepackage{amsthm}
\usepackage{thm-restate}
\usepackage{graphicx}
\usepackage{float}
\graphicspath{ {./pics/} }

\newtheorem{theorem}{Theorem}
\newtheorem{corollary}[theorem]{Corollary}
\newtheorem{lemma}[theorem]{Lemma}

\newtheorem{proposition}[theorem]{Proposition}

\newtheorem*{theorem*}{Theorem}
\newtheorem*{remark}{Remark}

\numberwithin{theorem}{section} 
\numberwithin{corollary}{section}
\numberwithin{lemma}{section}
\numberwithin{proposition}{section}
\numberwithin{definition}{section}

\begin{document}

\title{Boundary and Interior of Root Sets of Power Series}

\author{Jacob Kewarth}
\address{University of Toronto, 40 St. George Street, M5S 2E4, Toronto, Canada}
\email{jacob.kewarth@mail.utoronto.ca}

\begin{abstract}
In this paper we study the set $\mathcal{M}_M$ of roots of power series with coefficients in $\{0, \pm 1, ..., \pm M \}$ and constant coefficient $1$. We prove that every nonreal root in the unit disk of a polynomial of this form is an interior point of $\mathcal{M}_M$. Moreover we establish a dimension formula and show the corresponding Hausdorff measure is finite and positive for a family of attractors parameterized by boundary parameters $\lambda \in \partial \mathcal{M}_M \setminus \mathbb{R}$. Finally we show that $\partial \mathcal{M}_M$ has Lebesgue measure zero.
\end{abstract}

\maketitle


\section{Introduction}

Let $M \geqslant 1$ be an integer and let $\mathbb{D} = \{ |z| < 1 \}$ denote the unit disk. Define the sets:
\[
E_M = \{ -M, -(M-1), ..., 0, ..., M-1, M\},
\]
and:
\[
\mathcal{M}_M = \{ \lambda \in \mathbb{D} \mid 1 + \sum_{j=1}^\infty p_j \lambda^j = 0 \text{ for some } p_j \in E_M \}.
\]

Such sets have been well studied both for the classical case $M=1$ (see \cite{BHa}, \cite{Bo1}, \cite{Bo2}, \cite{Ba}, \cite{SX}, \cite{So2}, \cite{CKW}) and also for more general $M$ (see \cite{So1}, \cite{EJS}, \cite{EJ}). See also \cite{BHu}, \cite{Na}, \cite{BY}, \cite{HI}, \cite{Ke}, \cite{BBBP1}, \cite{BBBP2}, \cite{Ti} among others for related sets.

A central question has been to understand the interior points of $\mathcal{M}_M$. Since the intersection of $\mathcal{M}_M$ with a ball of sufficiently small radius yields two real intervals, interior points cannot be dense in $\mathcal{M}_M$. Still computer images suggest that interior points are dense in $\mathcal{M}_M \setminus \mathbb{R}$. For $M=1$ this was conjectured by Bandt \cite{Ba} and proven in \cite{CKW}.

A stronger conjecture is that all nonreal roots in $\mathbb{D}$ of polynomials with coefficients in $E_M$ and constant coefficient $1$ are interior points of $\mathcal{M}_M$. For the $M=1$ case \cite{SX} proved there is a region of $\mathcal{M}_M$ for which nonreal polynomial roots are interior points. This result was generalized to all $M$ in \cite{EJS} and \cite{EJ} where it was shown that for $M \geqslant 19$, all nonreal polynomial roots are interior points. For smaller $M$, including $M=1$, this remains unknown. In this paper we prove both conjectures for all $M \geqslant 1$.

These sets also admit a dynamical interpretation. Indeed, let:
\[
T_M = \{ 0, 1, ..., M \}.
\]
Note that $E_M = T_M - T_M$.

Then we define a family of iterated function systems (IFS) parameterized by $\lambda \in \mathbb{D}$ by:
\[
f_{s, \lambda}(z) = \lambda z + s ,\quad \text{where $s \in T_M$}.
\]

Since $\lambda \in \mathbb{D}$ these maps are contractions and therefore they have an attractor:
\[
A_{\lambda, M}  = \left\{ \sum_{j=0}^\infty b_j \lambda^j \mid b_j \in T_M \right\}.
\]
When $M$ is understood to be fixed we will often omit it and simply write $A_{\lambda}$.

Then it is well known that $A_{\lambda, M}$ is connected if and only if $\lambda \in \mathcal{M}_M$ (see \cite{EJS}), thus our root set is a connectedness locus for families of linear IFS. This parallels the relationship between the Mandelbrot set and Julia sets in holomorphic dynamics.

The goal of this paper is to leverage techniques from IFS theory to study the interior of $\mathcal{M}_M$ and prove the stronger conjecture for all $M$. Namely:

\begin{theorem}
\label{intthm}
Let $\lambda \in \mathcal{M}_M \setminus \mathbb{R}$. If $\lambda$ is the root of a polynomial $P(z) = 1 + \sum_{j=1}^N p_jz^j$ with $p_j \in E_M$ then $\lambda \in \operatorname{int}\mathcal{M}_M$.
\end{theorem}
As a consequence of theorem \ref{intthm}, every nonreal element of $\mathcal{M}_M$ lies in the closure of its interior.

In fact the main estimate we establish in this paper also gives a useful tool for studying the boundary of $\mathcal{M}_M$. Indeed in theorem~\ref{separationthm} we show for any nonreal boundary point $\lambda$ there is a constant $c_\lambda >0$ such that for any polynomial $P(z)$ of degree $N$ with coefficients in $E_M$ and constant coefficient $1$:
\[
|P(\lambda)| \geqslant c_\lambda |\lambda|^{N+1}.
\] 

This estimate has further applications for studying the geometry of the attractors $A_\lambda$ for $\lambda \in \partial \mathcal{M}_M \setminus \mathbb{R}$. In particular we further show:

\begin{theorem}
\label{hthm}
For any $\lambda \in \partial \mathcal{M}_M \setminus \mathbb{R}$, we have for $s(\lambda) = \frac{\log{(M+1)}}{-\log{|\lambda |}}$:
\[
0 < \mathcal{H}^{s(\lambda)}(A_\lambda) < \infty, \quad \dim_H(A_\lambda) = s(\lambda).
\]

Moreover, for such parameters the open set condition (OSC) holds.
\end{theorem}

Note this dimension formula is straightforward when $A_\lambda$ is totally disconnected, in which case $\lambda \notin \mathcal{M}_M$. On the other hand, for nonreal $\lambda$ whose modulus is close to $1$, $A_\lambda$ will have nonempty interior and so $\dim_H A_\lambda$ will be $2$.

It is interesting to compare theorem \ref{hthm} with the results in \cite{SX}. For $M=1$, they showed that for Lebesgue almost every $\lambda \in \mathcal{M}_1$ with $|\lambda| < 1/\sqrt{2}$, one has $\dim_H(A_\lambda) = s(\lambda)$ but $\mathcal{H}^{s(\lambda)}(A_\lambda) = 0$. Our result shows that, at every nonreal boundary parameter, the dimension formula still holds, but the corresponding Hausdorff measure is always positive.

For the dimension formula, \cite{SS}
subsequently proved that, for every fixed $M \geqslant 1$,
$\dim_H A_\lambda = \min\{2, s(\lambda)\}$ for all nonreal
$\lambda \in \mathbb{D}$ outside an exceptional set of zero
packing dimension.

For related families of fractal $n$-gons, \cite{BHu}
proved the OSC at every boundary parameter when $n > 2$, $n \neq 4$, and explicitly left the $n=2$ case open.
Note their $n=2$ connectedness locus is the same as our $\mathcal{M}_1$.

Combining the result of \cite{SX} with our theorem \ref{hthm} gives that $\partial \mathcal{M}_1$ has Lebesgue measure zero. In fact, the estimates developed here give that for all $M$:

\begin{theorem}
\label{thm3}
For every $M \geqslant 1$, the boundary $\partial \mathcal{M}_M$ has zero Lebesgue measure.
\end{theorem}

\medskip
\noindent\textbf{AI Disclosure.}
The author used ChatGPT (OpenAI) to assist with mathematical
discussions and editing during the preparation of this paper.
The final draft was entirely written by the author who takes full responsibility for its contents.

\section{Preliminaries}

Here we collect some basic properties about these families. First note that since $\mathcal{M}_M$ lives in the unit disk $\mathbb{D}$, anytime we reference $\partial \mathcal{M}_M$ we mean the boundary relative to $\mathbb{D}$ rather than relative to $\mathbb{C}$. Relative to $\mathbb{D}$ it is well known that $\mathcal{M}_M$ is closed.

From \cite{So1} we have:
\begin{lemma}
\label{annuluslemma}
$\mathcal{M}_M$ contains the annulus $\frac{1}{\sqrt{M+1}} \leqslant |\lambda| < 1$.
\end{lemma}

\begin{lemma}
\label{zeroarealemma}
If $|\lambda| < \frac{1}{\sqrt{M+1}}$ then the Lebesgue measure $m(A_{\lambda}) = 0$.
\end{lemma}

\begin{proof}
Since $A_\lambda = \bigcup_{j=0}^M (\lambda A_\lambda + j)$ we have that:
\[
m(A_\lambda) \leqslant (M+1)|\lambda|^2 m(A_\lambda).
\]
When $|\lambda| < \frac{1}{\sqrt{M+1}}$ we have $(M+1)|\lambda|^2 < 1$ so $m(A_\lambda) = 0$.
\end{proof}

\begin{remark}
The first lemma above gives that points on the boundary of $\mathcal{M}_M$ are all $\leqslant \frac{1}{\sqrt{M+1}}$ in modulus. On the other hand, the previous lemma requires a strict inequality, thus there could be points on the boundary whose attractors have positive area. In fact $\frac{\pm i}{\sqrt{M+1}}$ are two such points, and may very well be the only two (this was shown for $M=1$ in \cite{CW} but is open for general $M$). This potential for boundary points with $|\lambda| = \frac{1}{\sqrt{M+1}}$ will require separate arguments.
\end{remark}

The next result is due to \cite{BBBP1} for the $M=1$ case. Their proof easily generalizes to all $M$ so we quickly sketch the details.
\begin{lemma}
\label{simplelemma}
If $f(z) = 1 + \sum_{j=1}^\infty p_j z^j$ with $p_j \in E_M$ has a nonreal root $\lambda$ such that $|\lambda| \leqslant \frac{1}{\sqrt{M+1}}$ then the root is simple.
\end{lemma}

\begin{proof}
Since the coefficients of $f(z)$ are all real and $\lambda$ is nonreal we know its conjugate $\overline{\lambda}$ is also a root of $f(z)$ and that they both have the same multiplicity, call it $m$. Applying Jensen's formula gives:
\[
2m\log \frac{R}{|\lambda|} \leqslant \frac{1}{2\pi} \int_0^{2\pi} \log |f(Re^{it})| dt.
\]
Concavity implies:
\[
\frac{1}{2\pi} \int_0^{2\pi} \log |f(Re^{it})| dt \leqslant \frac{1}{2} \log(\frac{1}{2\pi} \int_0^{2\pi} |f(Re^{it})|^2dt).
\]
By Parseval's identity we have:
\[
\frac{1}{2\pi} \int_0^{2\pi} |f(Re^{it})|^2dt = 1 + \sum_{j=1}^\infty |p_j|^2 R^{2j} \leqslant 1 + \frac{M^2R^2}{1-R^2}.
\]
Thus we obtain:
\[
2m\log \frac{R}{|\lambda|} \leqslant \frac{1}{2} \log(1 + \frac{M^2R^2}{1-R^2}).
\]
Suppose we had $m \geqslant 2$. Then when $R^2 = \frac{4}{5}$ and $|\lambda| \leqslant \frac{1}{\sqrt{M+1}}$ we have:
\[
(\frac{4(M+1)}{5})^4 \leqslant 1 + 4M^2
\]
which is false for $M \geqslant 1$.

Hence our assumption that $m \geqslant 2$ was wrong so we must have $m = 1$ as wanted.
\end{proof}

Recall that the Hausdorff distance $d_H(K, L)$ of two nonempty compact sets is defined as:
\[
d_H(K, L) = \max \{ \sup_{x \in K} \operatorname{dist}(x, L), \sup_{y \in L} \operatorname{dist}(y, K) \}
\]
where $\operatorname{dist}(x, L) = \inf_{y \in L} |x-y|$ is the usual distance. Then a fairly well known fact about our attractor family is that the attractor depends continuously on $\lambda$ in the Hausdorff metric.

\begin{lemma}
\label{hlemma}
Fix $r \in (0, 1)$. Then for any $|\lambda|, |\eta| \leqslant r$ we have
\[
d_H(A_\lambda, A_\eta) \leqslant \frac{M}{(1-r)^2}|\lambda - \eta|.
\]
In particular, $\lambda \mapsto A_\lambda$ is continuous in Hausdorff distance.
\end{lemma}

\begin{proof}
Let $P(z) = \sum_{j=0}^\infty p_jz^j$ with $p_j \in T_M$. Then $P(\lambda) \in A_\lambda$ and $P(\eta) \in A_\eta$. Using difference of powers we get:
\[
|\lambda^j - \eta^j| \leqslant jr^{j-1}|\lambda - \eta|.
\]
Hence we get:
\[
|P(\lambda) - P(\eta)| \leqslant M|\lambda - \eta| \sum_{j=1}^\infty jr^{j-1} = \frac{M}{(1-r)^2} |\lambda-\eta|.
\]
Thus for any point in $A_\lambda$ there is a point in $A_\eta$ at distance at most $\frac{M}{(1-r)^2}|\lambda - \eta|$ away. Likewise, for any point in $A_\eta$ there is a point in $A_\lambda$ at distance at most $\frac{M}{(1-r)^2}|\lambda - \eta|$ away.
\end{proof}

\section{Difference Sets}

Consider the difference set $D_\lambda = A_\lambda - A_\lambda$. Notice that:
\[
D_\lambda = \left\{ \sum_{j=0}^\infty p_j\lambda^j \mid p_j \in E_M \right\}.
\]
The goal of this section is to prove that $D_\lambda$ contains a neighbourhood of zero whenever  $\operatorname{int} A_\lambda = \emptyset$, $A_\lambda$ is connected, and $\lambda$ is nonreal. Moreover we show that this neighbourhood is locally uniform. This is essentially the same approach as in \cite{SX}.

We begin by establishing some basic geometric lemmas.

\begin{lemma}
\label{gammalemma}
Let $\gamma: [0, 1] \rightarrow \mathbb{C}$ be a continuous map with endpoints satisfying:
\[
\operatorname{Im}(\gamma(0)) = \operatorname{Im}(\gamma(1)), \quad \operatorname{Re}(\gamma(0)) < \operatorname{Re}(\gamma(1)).
\]

Suppose further that for all $t \in (0, 1)$ we have $\operatorname{Im}(\gamma(t)) > \operatorname{Im}(\gamma(0))$. Then for any real number $s \in (0, \operatorname{Re}(\gamma(1)) - \operatorname{Re}(\gamma(0)))$ the images of $\gamma(t)$ and $\gamma(t) + s$ intersect.
\end{lemma}

\begin{proof}
Consider the half plane given by:
\[
H = \{ z \in \mathbb{C} \mid \operatorname{Im}(z) \geqslant \operatorname{Im}(\gamma(0)) \}.
\]
Since both endpoints of $\gamma$ lie on the boundary of $H$, we have that $H \setminus \gamma([0, 1])$ contains at least two connected components. By construction, we have that the endpoints of $\gamma(t) + s$ lie in different components, hence the images of $\gamma(t)$ and $\gamma(t) + s$ intersect.
\end{proof}

\begin{lemma}
\label{squarelemma}
Let $K \subseteq \mathbb{C}$ be a compact connected set. Let $0 < s < 2a$ be such that $K \cap (K + s) = \emptyset$. Consider a rectangle:
\[
Q = \{z \in \mathbb{C} \mid |\operatorname{Re}(z) - c| \leqslant a, |\operatorname{Im}(z) - h| \leqslant b \}.
\]
If $Q \cap K = \emptyset$ then $K$ contains at most one of $z_- = x_- + iy_-$ and $z_+ = x_+ + iy_+$ where:
\[
x_- < c-a, |y_- - h| < b, \quad \text{and } \quad  x_+ > c+a, |y_+ - h| < b.
\]
\end{lemma}

\begin{proof}
Suppose we had two such points $z_-$ and $z_+$ both in $K$. Since $K$ and $Q$ are compact and $K$ is disjoint from both $K + s$ and $Q$ we can pick $\delta > 0$ such that:

\[
2\delta < \operatorname{dist}(K, K+s), \quad \text{and} \quad \delta < \operatorname{dist}(K,Q).
\]

Let $U$ be the $\delta$ neighbourhood of $K$. Since $K$ is connected, so is $U$. Moreover, we have:
\[
U \cap Q = \emptyset, \quad \text{and} \quad U \cap (U + s) = \emptyset.
\]
Choose a path $\gamma(t)$ in $U$ such that $\gamma(0) = z_-$ and $\gamma(1) = z_+$. Consider the sets:
\[
L_{\pm} = \{ z \in \mathbb{C} \mid \pm (\operatorname{Re}(z) - c) \geqslant a, |\operatorname{Im}(z) - h| \leqslant b \}.
\]

Notice that by construction $z_- \in L_-$ and $z_+ \in L_+$. In particular, this means $\gamma(0) \in L_-$ and $\gamma(1) \in L_+$. Define $t_+ = \inf\{ t \in [0, 1] \mid \gamma(t) \in L_+ \}$ and $t_- = \sup\{ t \in [0, t_+] \mid \gamma(t) \in L_-\}$.

For $t \in (t_-, t_+)$, $\gamma(t)$ is not contained in $L_- \cup L_+$. Since every element of $U$ in the closed strip belongs to one of these sets, $\gamma(t)$ lies outside the closed strip. Hence continuity implies that the whole curve $\gamma(t)$ for $t \in (t_-, t_+)$ lies either above or below the strip. Both endpoints must then lie on the same boundary component of the strip. Notice also that $|\operatorname{Re}(\gamma(t_+)) - \operatorname{Re}(\gamma(t_-))| \geqslant 2a > s$.

Thus by applying lemma~\ref{gammalemma} to the restricted path $\gamma \mid_{[t_-, t_{+} ]}$ (or its reflection), we have $\gamma([t_-, t_+]) \cap (\gamma([t_-, t_+]) + s) \neq \emptyset$. This contradicts $U \cap (U + s) = \emptyset$ and so both of $z_-$ and $z_+$ cannot lie in $K$ simultaneously.
\end{proof}

\begin{lemma}
\label{convexlemma}
Let $K_n, K \subseteq \mathbb{C}$ be compact connected sets such that the $K_n$ converge to $K$ in Hausdorff distance. Let $v_n, u \in \mathbb{C}$ be such that:
\begin{enumerate}
\item $v_n \neq 0$ and $v_n \rightarrow 0$,
\item $\frac{v_n}{|v_n|} \rightarrow u$ for some $u \in \mathbb{C}$,
\item $K_n \cap (K_n + v_n) = \emptyset$.
\end{enumerate}

Then any line parallel to $u$ meets $K$ in an interval or a point, or is disjoint from $K$.
\end{lemma}

\begin{proof}
Let $u_n = \frac{v_n}{|v_n|}$ and define the rotations $E_n = \overline{u_n} \cdot K_n$, $E = \overline{u} \cdot K$. Notice that $E_n \rightarrow E$ in Hausdorff distance. Indeed, by the triangle inequality we have:
\begin{align*}
d_H(E_n, E) &\leqslant d_H(\overline{u_n} K_n, \overline{u_n}K) + d_H(\overline{u_n}K, \overline{u}K) \\
&\leqslant d_H(K_n, K) + |u_n - u| \max_{z \in K} |z| \\
&\rightarrow 0.
\end{align*}

Since we have:
\[
K_n \cap (K_n + v_n) = \emptyset
\]
this implies:
\[
E_n \cap (E_n + \overline{u_n}v_n) = \emptyset
\]
also. Notice that $s_n = \overline{u_n}v_n \in \mathbb{R}$.

Any line parallel to $u$ becomes a horizontal line when multiplied by $\overline{u}$, hence it suffices to understand the intersection of $E$ with horizontal lines. Suppose that the intersection of $E$ with some horizontal line $L(t) = t + ih$ with constant imaginary part $h$ is not empty, not a point, and not an interval. That means that there are points $(a, h)$ and $(b, h) \in E$ such that for any $c \in (a, b)$, $c + ih \notin E$. Since $E$ is compact, this implies there is some closed rectangle $Q$ disjoint from $E$ with $(a, h)$ and $(b, h)$ on opposite sides of $Q$. Then for all $n$ large, $E_n$ must also be disjoint from $Q$ and must have points $a_n +ih_n$ and $b_n+ih'_n$ on opposite sides of $Q$ such that $a_n+ih_n \rightarrow a+ih$ and $b_n+ih'_n \rightarrow b+ih$.

Since $\overline{u_n}v_n \rightarrow 0$, lemma~\ref{squarelemma} implies for all $n$ large, $E_n \cap (E_n + \overline{u_n}v_n) \neq \emptyset$ giving a contradiction.
\end{proof}

The previous results were quite general but now we apply them to our particular family of attractors.

\begin{lemma}
\label{intlemma}
Let $\lambda \in \mathbb{D} \setminus \mathbb{R}$ and suppose $A_{\lambda}$ is connected. Suppose there is some $u \in \mathbb{C}$, $u \neq 0$, such that the intersection of any line parallel to $u$ with $A_{\lambda}$ is connected or empty. Then $A_\lambda$ has nonempty interior.
\end{lemma}

\begin{proof}
Fix some $u \in \mathbb{C}$ and assume $|u| = 1$. Consider a line $l$ parallel to $u$. Possibly rotating $A_\lambda$, $T_M$, and $u$, we can take the line to be horizontal.

First suppose that $l$ intersects $A_\lambda$ in an interval containing more than a single point, say $I = \{ x + i h \mid a \leqslant x \leqslant b \}$ for some $a < b$.

Since $\lambda \cdot A_\lambda \subseteq A_\lambda$ and $\lambda$ is not real, we have that $J = \lambda \cdot I$ is a nonhorizontal line segment contained in $A_\lambda$.

Write $J = \{ (j(t), t) \mid \alpha \leqslant t \leqslant \beta \}$, $\alpha < \beta$. Consider the midpoint $z_0 = (j(\frac{\alpha + \beta}{2}), \frac{\alpha + \beta}{2})$. Since $z_0 \in A_\lambda$ we can write $z_0 = \sum_{j=0}^\infty p_j\lambda^j$ for some $p_j \in T_M$. Let $P_n(\lambda) = \sum_{j=0}^{n} p_j\lambda^j$. Then for any $n$ we have $z_0 \in C_n := P_n(\lambda) + \lambda^{n+1}A_\lambda \subseteq A_\lambda$. Since $|\lambda|^n \operatorname{diam}(A_\lambda)$ tends to $0$, we can choose $N$ large enough so that $\operatorname{diam}(C_N) < \frac{\beta - \alpha}{2}$. Notice that since $0, 1, \lambda \in A_\lambda$, we have that $A_\lambda$ is not collinear, hence each set $C_n$ is not collinear either. In particular, this means that there is some $z_1 \in C_N$ such that $z_1 \notin L$ where $L =  \{ (j(t), t) \mid t \in \mathbb{R} \}$ is the bi-infinite line containing $J$. 

Since $z_1 \in C_N$ we can write $z_1 = \sum_{j=0}^\infty q_j\lambda^j$ for some $q_j \in T_M$ such that $q_0 = p_0, ..., q_N = p_N$. Write $Q_n(\lambda) = \sum_{j=0}^n q_j\lambda^j$ and $B_n = Q_n(\lambda) + \lambda^{n+1} A_\lambda$. Thus $B_N = C_N$ so for all $N' > N$ we have $B_{N'} \subseteq C_N$. Choose $N'$ large enough so that $z_0 \notin B_{N'}$ and:
\[
\operatorname{diam}(B_{N'}) < \operatorname{dist}(z_1, L).
\]
Then we have $L \cap B_{N'} = \emptyset$. Also since $B_{N'} \subseteq C_N$, $z_0 \in C_N$, and $\operatorname{diam}(B_{N'}) \leqslant \operatorname{diam}(C_N) < \frac{\beta - \alpha}{2}$ we have $\operatorname{Im}(B_{N'}) \subseteq \{ \alpha < y < \beta \}$.

Consider the map $\Delta(x, t) = x - j(t)$. Since $B_{N'}$ is compact and $B_{N'} \cap L = \emptyset$, we have:
\[
e = \min_{ x+it \in B_{N'}} |\Delta(x, t)| > 0.
\]

Also, since $A_\lambda$ is connected it follows $B_{N'}$ is connected, hence $\Delta(x, t)$ must have a constant sign on $B_{N'}$. Suppose it is positive. Then we have $x \geqslant j(t) + e$ for any $x + it \in B_{N'}$. Let $\operatorname{Im}(B_{N'}) = [\alpha', \beta']$. Notice $\alpha' < \beta'$ as otherwise $B_{N'}$ would be collinear. For each $t \in \operatorname{Im}(B_{N'})$ we can choose $x_t$ such that $x_t + it \in B_{N'}$.

Then we know that $B_{N'}$ contains $x_t + it$ and hence $A_\lambda$ contains $x_t + it$. On the other hand, it also contains $j(t) + it$ and $x_t \geqslant j(t) + e$. Then since we have assumed all horizontal lines intersect $A_\lambda$ in either an interval, a point, or an empty intersection, it follows $A_\lambda$ contains every point $\{ x+it \mid j(t) \leqslant x \leqslant x_t \}$. This holds for all $t \in [\alpha', \beta']$. Hence $A_\lambda$ contains at least the parallelogram $\{ j(t) + r + it \mid \alpha' \leqslant t \leqslant \beta', 0 \leqslant r \leqslant e \}$. This implies $A_\lambda$ has nonempty interior as claimed.

The above showed that a single intersection which was an interval implied $A_\lambda$ contained an interior point. It remains to rule out the case in which every horizontal line meets $A_\lambda$ in at most one point. Suppose this is the case. Consider the projection $\pi(x + iy) = y$. This map is always a continuous surjection onto its image and if horizontal lines only meet $A_\lambda$ in at most one point, this would make $\pi \mid_{A_\lambda}$ a bijection onto $\pi(A_\lambda)$. As $A_\lambda$ is a compact connected set, $\pi(A_\lambda)$ is an interval. Since multiplication by $\lambda$ sends $A_\lambda$ into itself, we have that the composition:
\[
f = \pi \circ (z \mapsto \lambda z) \circ (\pi \mid_{A_\lambda})^{-1}
\]
is a continuous injective map from an interval into itself. In particular this means it must be strictly monotonic. By taking the composition of $f$ with itself we have a map $f^{\circ 2}$ which is strictly monotonically increasing. Since $0 \in A_\lambda$ and $\pi(0) = 0$ we know $f$ (and hence $f^{\circ 2}$) fixes $0$. Repeating we get the same is true for all even compositions:
\[
f^{\circ 2n} = \pi \circ (z \mapsto \lambda^{2n}z) \circ (\pi\mid_{A_\lambda})^{-1}.
\]

In particular this means that for any $z \in A_\lambda$, $z \neq 0$, the sign of $\pi(\lambda^{2n}z)$ is constant which is impossible for $\lambda \notin \mathbb{R}$.
\end{proof}

\begin{proposition}
\label{prop1}
Let $\lambda \in \mathcal{M}_M \setminus \mathbb{R}$ such that $\operatorname{int} A_\lambda= \emptyset$. Then there exist $\epsilon, \delta > 0$ such that for every $\eta \in \mathcal{M}_M$, if $|\eta - \lambda| < \delta$ then $B(0, \epsilon) = \{ |z| < \epsilon \} \subseteq D_{\eta}$.
\end{proposition}

\begin{proof}
If the statement were false then for all $n$ we can choose $\eta_n \in \mathcal{M}_M$ such that $|\eta_n - \lambda| < \frac{1}{n}$ and $v_n$ such that $0 < |v_n| < 1/n$, $v_n \notin D_{\eta_n}$.

Then we have $A_{\eta_n} \cap (A_{\eta_n} + v_n) = \emptyset$, and $A_{\eta_n}$ converges to $A_{\lambda}$ in Hausdorff distance by lemma~\ref{hlemma}.

Taking a subsequence if needed we can assume the ratios $\frac{v_n}{|v_n|}$ converge to some $u$. Then by lemma~\ref{convexlemma}, $A_\lambda$ satisfies the preconditions for lemma~\ref{intlemma} contradicting the assumption that $A_\lambda$ had empty interior.
\end{proof}

\begin{corollary}
\label{cor1}
If $K \subseteq \mathcal{M}_M \setminus \mathbb{R}$ is a nonempty compact set such that for all $\lambda \in K$, $\operatorname{int} A_\lambda = \emptyset$ then there exist $\delta, \epsilon > 0$ such that for all $\eta \in \mathcal{M}_M$, if $\operatorname{dist}(\eta, K) < \delta$ then $B(0, \epsilon) \subseteq D_\eta$.
\end{corollary}

\begin{proof}
Applying proposition~\ref{prop1} to every point in $K$ yields pairs $\epsilon_\lambda$ and $\delta_\lambda$ such that if $|\eta - \lambda| < \delta_\lambda$ then $B(0, \epsilon_\lambda) \subseteq D_\eta$. By compactness of $K$ there are finitely many points $\lambda_1, \lambda_2, ..., \lambda_n$ such that the balls $B(\lambda_i, \delta_{\lambda_i})$ cover $K$. Let $U = \bigcup_{i=1}^n B(\lambda_i, \delta_{\lambda_i})$.

Choose $\delta = \operatorname{dist}(K, \mathbb{C} \setminus U)$ and $\epsilon = \min_i \epsilon_{\lambda_i}$. Then the result immediately follows.
\end{proof}

\section{Separation Estimates}

In this section we combine the locally uniform interior region established for difference sets in corollary~\ref{cor1} with a standard Rouch\'e argument to get a criterion for nonreal polynomial roots being interior points. This criterion then allows us to prove a separation estimate for nonreal boundary points.

\begin{proposition}
\label{intprop}
Let $K \subseteq \mathcal{M}_M \setminus \mathbb{R}$ be a nonempty compact set such that for all $\lambda \in K$, $\operatorname{int} A_\lambda = \emptyset$. For any $b > 0$ there exist $\epsilon, \rho > 0$ and an integer $N_0 > 0$ such that for any $N > N_0$ and for any polynomial $P(z) = 1 + \sum_{j=1}^N p_jz^j$ with $p_j \in E_M$, if:
\begin{enumerate}
\item $\operatorname{dist}(\mu, K) < \rho$,
\item $|P'(\mu)| > b$,
\item $|P(\mu)| < \epsilon |\mu|^{N+1}$
\end{enumerate}
then $\mu \in \operatorname{int}\mathcal{M}_M$.
\end{proposition}

\begin{proof}
Choose $\epsilon, \delta > 0$ as in corollary~\ref{cor1} applied to $K$. Choose $r \in (\max_K |\lambda|, 1)$ and define the three numbers:
\[
L_0 = \frac{M}{1-r}, \quad L_1 = \frac{M}{(1-r)^2}, \quad L_2 = \frac{2M}{(1-r)^3}.
\]
For all $|z| < r$, any power series $H(z) = \sum_{j=0}^{\infty} h_j z^j$ with $h_j \in E_M$ satisfies $|H(z)| \leqslant L_0$ and $|H'(z)| \leqslant L_1$. Moreover note $|P''(z)| \leqslant L_2$.

Choose:
\[
\rho = \min\{ \frac{\delta}{2}, \frac{r - \max_K |\lambda|}{3}, \frac{1}{2} \min_K |\lambda|, \frac{b}{2L_2} \}
\]
and $N_0$ such that for any $N \geqslant N_0$ we have:
\[
L_1 r^{N+1} < \frac{b}{8}, \quad (N+1)L_0r^N + L_1r^{N+1} < \frac{b}{4}.
\]

Suppose that $\mu \notin \mathcal{M}_M$. Then choose $\lambda \in \mathcal{M}_M$ with: 
\[
|\lambda - \mu| = \operatorname{dist}(\mu, \mathcal{M}_M) > 0.
\]
Notice $\operatorname{dist}(\mu, \mathcal{M}_M) \leqslant \operatorname{dist}(\mu, K) < \rho$. Hence $\operatorname{dist}(\lambda, K) < 2\rho \leqslant \delta$.

Since $|\frac{P(\mu)}{\mu^{N+1}}| < \epsilon$, by corollary~\ref{cor1} it belongs to $D_\lambda$ so we can choose $H(z) = \sum_{j=0}^\infty h_jz^j$ such that $H(\lambda) = -\frac{P(\mu)}{\mu^{N+1}}$.

Consider the series $F(z) = P(z) + z^{N+1}H(z)$. Note it has coefficients in $E_M$. We have that:
\[
|F(\mu)| = |\mu^{N+1}||H(\mu) - H(\lambda)| \leqslant L_1r^{N+1} \operatorname{dist}(\mu, \mathcal{M}_M) < b\frac{\operatorname{dist}(\mu, \mathcal{M}_M)}{8}.
\]

The disk $|z-\mu| \leqslant \frac{\operatorname{dist}(\mu, \mathcal{M}_M)}{2}$ is contained in $|z| < r$. For such $z$ we have:
\begin{align*}
|F'(z) - P'(\mu)| &= |P'(z) - P'(\mu) + (N+1)z^N H(z) + z^{N+1}H'(z)| \\
&\leqslant |P'(z) - P'(\mu)| + (N+1)|z|^N|H(z)| + |z|^{N+1}|H'(z)| \\
&\leqslant L_2 \frac{\operatorname{dist}(\mu, \mathcal{M}_M)}{2} + (N+1)L_0r^{N} + L_1r^{N+1} \\
&< \frac{b}{2}.
\end{align*}

On the circle $|z-\mu| = \frac{\operatorname{dist}(\mu, \mathcal{M}_M)}{2}$ we have:
\begin{align*}
|F(z) - P'(\mu)(z-\mu)| &= |F(\mu) + (z-\mu)\int_0^1 F'(\mu+t(z-\mu)) - P'(\mu)dt| \\
&\leqslant |F(\mu)| + |z-\mu| \frac{b}{2} \\
&< b\frac{\operatorname{dist}(\mu, \mathcal{M}_M)}{8} + b\frac{\operatorname{dist}(\mu, \mathcal{M}_M)}{4} \\
&= b \frac{3 \operatorname{dist}(\mu, \mathcal{M}_M)}{8}.
\end{align*}

Since we know $|P'(\mu)| > b$ this gives:
\[
|F(z) - P'(\mu)(z-\mu)|  < |P'(\mu)(z-\mu)|.
\]

Then by Rouch\'e, $F$ must have a root in $B(\mu, \frac{\operatorname{dist}(\mu, \mathcal{M}_M)}{2})$. In particular, this means $\mathcal{M}_M$ contains an element closer to $\mu$ than $\lambda$, a contradiction.

Therefore our assumption was wrong and $\mu \in \mathcal{M}_M$.

Notice that for fixed $N$ and $P$ the elements $\mu$ satisfying (1), (2), and (3) form an open set. Hence, any such $\mu$ belongs to $\mathcal{M}_M$ and thus must be an interior point.
\end{proof}

\begin{lemma}
\label{ktlemma}
Let $K \subseteq \mathcal{M}_M \setminus \mathbb{R}$ be a nonempty compact set. Suppose for every $\lambda \in K$ we have $|\lambda| \leqslant \frac{1}{\sqrt{M+1}}$. Then there exist some $\kappa, \tau > 0$ such that for any power series $P(z) = 1 + \sum_{j=1}^\infty p_jz^j$ with $p_j \in E_M$, for any $\lambda \in K$, if $|P(\lambda)| \leqslant \tau$ then $|P'(\lambda)| \geqslant \kappa$.
\end{lemma}

\begin{proof}
If this were not true then we would be able to find some sequence $\lambda_n \in K$ all with $|\lambda_n| \leqslant \frac{1}{\sqrt{M+1}}$ and a sequence of power series $P_n(z)$ with coefficients in $E_M$ such that:
\[
|P_n(\lambda_n)| \leqslant \frac{1}{n}, \quad |P_n'(\lambda_n)| \leqslant \frac{1}{n}.
\]

Passing to subsequences if needed we can find a power series $P(z)$ such that the coefficients of $P_n(z)$ converge to the coefficients of $P(z)$ and we can find a point $\lambda$ such that $\lambda_n \rightarrow \lambda$. Note that since $\lambda_n \in K$ and $K$ is compact, $\lambda \in K$, so in particular the limit $\lambda$ is nonreal. Then it follows $\lambda$ is a double root of $P(z)$, contradicting lemma~\ref{simplelemma}.
\end{proof}

We now give a lower bound for polynomial values at boundary points $\lambda$.

\begin{proposition}
\label{boundryprop}
Let $K \subseteq \partial \mathcal{M}_M \setminus \mathbb{R}$ be a nonempty compact set. Suppose for every $\lambda \in K$ we have $\operatorname{int} A_\lambda = \emptyset$. Then there exists a constant $c_K > 0$ such that for any $\lambda \in K$ and every $N \geqslant 0$, if $P(z) = 1+\sum_{j=1}^N p_jz^j$ is a polynomial with $p_j \in E_M$ then:
\[
|P(\lambda)| \geqslant c_K |\lambda|^{N+1}.
\]
\end{proposition}

\begin{proof}
Apply lemma~\ref{ktlemma} and proposition~\ref{intprop} with $b = \frac{\kappa}{2}$.

Let $R = \max_K |\lambda|$. Notice $R \leqslant \frac{1}{\sqrt{M+1}}$. Increase $N_0$ if needed such that $\epsilon R^{N+1} \leqslant \tau$ for all $N \geqslant N_0$. If $|P(\lambda)| < \epsilon |\lambda|^{N+1}$ were true then by lemma~\ref{ktlemma} we would have $|P'(\lambda)| \geqslant \kappa > b$. But proposition~\ref{intprop} then implies $\lambda$ is an interior point, contradicting $K \subseteq \partial \mathcal{M}_M$.

The above proves the claim for all $N$ large. To get the claim for smaller $N$ replace any polynomial $P(z) = 1 + \sum_{j=1}^N p_jz^j$ with $1 + \sum_{j=1}^N p_jz^j + \sum_{j=N+1}^{N_0+1} 0 z^j$ and apply the above. This yields:
\[
|P(\lambda)| \geqslant \epsilon |\lambda|^{N_0+2} = \epsilon |\lambda|^{N_0 +1 - N} |\lambda|^{N+1} \geqslant \epsilon |\lambda|^{N_0+1} |\lambda|^{N+1}.
\]
Thus taking $c_K = \epsilon \min_{\lambda \in K} |\lambda|^{N_0+1}$ works for all $N$.
\end{proof}

\begin{theorem}
\label{separationthm}
For any $\lambda \in \partial \mathcal{M}_M \setminus \mathbb{R}$ there exists $c_\lambda > 0$ such that for any $N \geqslant 0$ and every $P(z) = 1 + \sum_{j=1}^N p_jz^j$ with $p_j \in E_M$, $|P(\lambda)| \geqslant c_\lambda |\lambda|^{N+1}$.
\end{theorem}

\begin{proof}
Fix $\lambda$. If $m(A_{\lambda}) = 0$ then this follows from proposition~\ref{boundryprop}. By lemma~\ref{zeroarealemma} we know when $|\lambda| < \frac{1}{\sqrt{M+1}}$ that $m(A_{\lambda}) = 0$ so it suffices to consider only $|\lambda| = \frac{1}{\sqrt{M+1}}$ and $m(A_{\lambda}) > 0$.

For each $\epsilon > 0$ we can define the open $\epsilon$ neighbourhood of $A_\lambda$ by $B_\epsilon = \{ z \mid \operatorname{dist}(z, A_\lambda) < \epsilon \}$. Choosing $\epsilon$ small enough we can assume $m(B_\epsilon) < 2m(A_\lambda)$.

For all $|v| < \epsilon$ both $A_\lambda$ and $A_\lambda + v$ are contained in $B_\epsilon$. Hence:
\begin{align*}
m(A_\lambda \cap (A_\lambda+v)) &= m(A_\lambda) + m(A_\lambda+v) - m(A_\lambda \cup (A_\lambda+v))\\
&\geqslant 2m(A_\lambda) - m(B_\epsilon)\\
&> 0.
\end{align*}
Recall that for any $n$ we can write:
\[
A_\lambda = \bigcup_{b_0, ..., b_{n-1} \in T_M} \left(\sum_{j=0}^{n-1} b_j\lambda^j + \lambda^n A_\lambda \right).
\]

There are $(M+1)^n$ sets above and each has area $|\lambda|^{2n} m(A_\lambda) = (M+1)^{-n} m(A_\lambda)$. Hence the intersection of any two sets corresponding to distinct choices $b_0, ..., b_{n-1}$ in the union has zero measure.

Consider now a polynomial $P(z) = 1 + \sum_{j=1}^N p_jz^j$ with $p_j \in E_M$. Since $E_M = T_M - T_M$ we can write $P$ as a difference of two polynomials $B(z) - C(z)$ with coefficients $b_j$ and $c_j$ in $T_M$. We can assume $b_0 = 1$ and $c_0 = 0$. Moreover for $j \geqslant 1$ we can choose $b_j = \max(p_j, 0)$ and $c_j = \max(-p_j, 0)$. Then the sets $B(\lambda) + \lambda^{N+1}A_\lambda$ and $C(\lambda) + \lambda^{N+1}A_\lambda$ occur in the above union, hence their intersection has measure zero also.

We can subtract  $C(\lambda)$ from both sets and scale by $\lambda^{-(N+1)}$ to get:
\[
m(A_\lambda \cap (A_\lambda + \frac{P(\lambda)}{\lambda^{N+1}})) = 0.
\]

However if $|\frac{P(\lambda)}{\lambda^{N+1}}| < \epsilon$ then this contradicts the above. Thus for $\epsilon$ sufficiently small we must have:
\[
|P(\lambda)| \geqslant \epsilon | \lambda |^{N+1}. \qedhere
\]
\end{proof}

Finally we obtain the first main result of the paper:

\begin{proof}[Proof of theorem~\ref{intthm}]
If this were not the case then $\lambda$ would be a boundary point. Thus by lemma~\ref{annuluslemma} we would have $|\lambda| \leqslant \frac{1}{\sqrt{M+1}}$. Hence by theorem~\ref{separationthm} we would have:
\[
|P(\lambda)| \geqslant c_\lambda |\lambda|^{N+1}
\]
which contradicts $P(\lambda) = 0$.
\end{proof}

\section{Hausdorff Measure and Dimension}

In this section we compute the Hausdorff dimension of $A_\lambda$ for $\lambda \in \partial \mathcal{M}_M \setminus \mathbb{R}$. We begin with a technical lemma to handle bounds for polynomials with starting coefficients other than $1$. The proof is quite similar to the proof of lemma~\ref{simplelemma}.
\begin{lemma}
\label{BGhelp}
Fix $| \lambda| \leqslant \frac{1}{\sqrt{M+1}}$, $\lambda \notin \mathbb{R}$. Let $M \geqslant 2$. Then there exists a constant $a > 0$ such that for any polynomial $P(z) = d + \sum_{j=1}^N p_jz^j$ with $d, p_j \in E_M$ and $d \geqslant 2$:
\[
|P(\lambda)| \geqslant a.
\]
\end{lemma}

\begin{proof}
First suppose for some power series $P(z) = d + \sum_{j=1}^\infty p_jz^j$ we had $P(\lambda) = 0$ with $d, p_j \in E_M$ and $d \geqslant 2$. Then for $0 < R < 1$, Jensen's formula, concavity of $\log$, and Parseval's identity give:
\[
2\log\frac{R}{|\lambda|} \leqslant \frac{1}{2} \log(1 + \frac{M^2R^2}{d^2(1-R^2)}).
\]
Taking $R = \frac{1}{\sqrt{2}}$ gives:
\[
\frac{(M+1)^2}{4} \leqslant \frac{1}{4|\lambda|^4} \leqslant 1 + \frac{M^2}{d^2} \leqslant 1 + \frac{M^2}{4}.
\]
This implies $2M+1 \leqslant 4$ which is false as $M \geqslant 2$. Thus no such power series exists.

If the statement were false then we could find a sequence $P_n(z)$ such that $|P_n(\lambda)| \rightarrow 0$. This would imply that, after possibly passing to a subsequence, we can find a limit $P(z)$ such that $P(\lambda) = 0$. This contradicts the above, hence there must exist some $a>0$ such that $|P(\lambda)| \geqslant a$ for all such polynomials $P$. 
\end{proof}

Next we give a specialized form of a result found in \cite{BG}.

\begin{lemma}
\label{BGlemma}
Fix $|\lambda| \leqslant \frac{1}{\sqrt{M+1}}$, $\lambda$ nonreal, and let $s(\lambda) = \frac{\log(M+1)}{-\log |\lambda|}$. Suppose there exists a constant $c > 0$ such that for any polynomial $P(z) = 1 + \sum_{j=1}^N p_jz^j$ with $p_j \in E_M$ we have:
\[
|P(\lambda)| \geqslant c |\lambda|^{N+1}.
\]
Then we have:
\[
0 < \mathcal{H}^{s(\lambda)}(A_\lambda) < \infty, \quad \dim_H A_\lambda = s(\lambda).
\]
Moreover in this case OSC holds for our IFS.
\end{lemma}

\begin{proof}
When $M > 1$ let $C$ be the minimum of $c$ and of the $a$ from lemma~\ref{BGhelp}. When $M = 1$ take $C = c$. Let $L = \frac{M}{1-|\lambda|}$. Consider two distinct polynomials:
\[
A(z) = \sum_{j=0}^{n-1} a_jz^j, \quad B(z) = \sum_{j=0}^{n-1} b_jz^j, \quad a_j, b_j \in T_M.
\]
Assume the two polynomials first differ at index $k$ and, after possibly relabeling, assume that $d = a_k-b_k > 0$. Thus:
\[
A(z) - B(z) = z^k(d + \sum_{j=1}^{n-1-k} d_jz^j), \quad d, d_j \in E_M.
\]
Then we have:
\[
|A(\lambda) - B(\lambda)| \geqslant C |\lambda|^{n}.
\]
In particular this means no nonzero polynomial with coefficients in $E_M$ has $\lambda$ as a root. Consider:
\[
S_n = \{ \sum_{j=0}^{n-1} s_j\lambda^j \mid s_j \in T_M \}.
\]
Different choices of coefficients $s_j$ yield different values. Thus $|S_n| = (M+1)^n$. Also notice $S_n \subseteq A_\lambda$.

Fix $0 < r < 1$ and consider $m$ such that $|\lambda|^m \leqslant r < |\lambda|^{m-1}$. Fix a base point $z$ and consider the ball $B(z, r)$. For any $n \geqslant m$, any point:
\[
x = \sum_{j=0}^{n-1} s_j \lambda^j \in S_n \cap B(z, r)
\]
can be truncated to give a point
\[
x_0 = \sum_{j=0}^{m-1} s_j\lambda^j \in S_m.
\]
Since
\[
|x - x_0| \leqslant L |\lambda|^m \leqslant L r
\]
we have $x_0 \in B(z, (1 + L)r)$. The different values of $x_0$ that occur this way (for a fixed $m$) must all be at least $Cr|\lambda|$ apart. Thus the disks centred at each $x_0$ of radius $\frac{Cr|\lambda|}{2}$ are disjoint and they lie inside $B(z, (1 + L + \frac{C|\lambda|}{2})r)$. Comparing areas gives an upper bound on the number of such $x_0$ as $U_0 = (1 + \frac{2(1+L)}{C|\lambda|})^2$.

Given $m$ coefficients there are $(M+1)^{n-m}$ extensions to $n$ coefficients. Hence:
\[
| S_n \cap B(z, r)| \leqslant U_0 (M+1)^{n-m}.
\]

Cover $A_\lambda$ by a finite number of balls $B(z_i, r_i)$, $r_i < 1$, $i=1, 2, ..., N$. Choose $m_i$ so that $|\lambda|^{m_i} \leqslant r_i < |\lambda|^{m_i-1}$ and let $n \geqslant \max_i m_i$. Then the above implies:
\[
(M+1)^n \leqslant \sum_{i=1}^N |S_n \cap B(z_i, r_i)| \leqslant U_0 \sum_{i=1}^N (M+1)^{n-m_i}.
\]
Dividing both sides by $(M+1)^n$ gives:
\[
1 \leqslant U_0 \sum_{i=1}^N (M+1)^{-m_i} .
\]
Since we can write $(M+1) = |\lambda|^{-s(\lambda)}$, using the fact that $|\lambda|^{m_i} \leqslant r_i$ this gives:
\[
1 \leqslant U_0 \sum_{i=1}^N r_i^{s(\lambda)}.
\]

Since $A_\lambda$ is compact the above also holds for a countable covering provided the radii of the balls are all strictly less than $1$. Hence we get:
\[
\mathcal{H}^{s(\lambda)}(A_\lambda) > 0.
\]

The fact that $\mathcal{H}^{s(\lambda)}(A_\lambda)$ is finite is fairly well known. Indeed, let $D = \operatorname{diam}(A_\lambda)$. For each $n$ we have:
\[
A_\lambda = \bigcup_{x \in S_n} (x + \lambda^n A_\lambda). 
\]
This is a covering of $A_\lambda$ by $(M+1)^n$ sets of diameter $|\lambda|^n D$. Since $(M+1)^n(|\lambda|^nD)^{s(\lambda)} = D^{s(\lambda)}$, we have $\mathcal{H}^{s(\lambda)}(A_\lambda) < \infty$.

Since $0 < \mathcal{H}^{s(\lambda)}(A_\lambda) < \infty$ we get $\dim_H A_\lambda = s(\lambda)$. Moreover since $\mathcal{H}^{s(\lambda)}(A_\lambda) > 0$, \cite{Sc} implies the IFS satisfies the OSC.
\end{proof}

\begin{proof}[Proof of theorem~\ref{hthm}]
By lemma~\ref{annuluslemma} any boundary point must satisfy $|\lambda| \leqslant \frac{1}{\sqrt{M+1}}$. Thus the theorem follows by combining theorem~\ref{separationthm} and lemma~\ref{BGlemma}.
\end{proof}

\section{Lebesgue Measure of the Boundary}

Before proving our final theorem we need one more technical lemma:

\begin{lemma}
\label{m0lemma}
Let $\lambda \in \partial \mathcal{M}_M \setminus  \mathbb{R}$ and $|\lambda | < \frac{1}{\sqrt{M+1}}$.

Then there exists $a, C > 0$ and an integer $N_0$ such that for any $N > N_0$, there is a $\lambda_N$ satisfying:
\[
|\lambda_N - \lambda | < C|\lambda|^{N+1}, \quad B(\lambda_N, a|\lambda|^{N+1}) \subseteq \operatorname{int} \mathcal{M}_M.
\]
\end{lemma}

\begin{proof}
Let $P(z) = 1 + \sum_{j=1}^\infty p_jz^j$ with $p_j \in E_M$ and $P(\lambda) = 0$. Consider $P_N(z) = 1 + \sum_{j=1}^N p_jz^j$ and consider a root $\lambda_N$ of $P_N$ such that $\lambda_N \rightarrow \lambda$.

By lemma \ref{simplelemma} we have $d = |P'(\lambda)| \neq 0$.

As $|\lambda| < \frac{1}{\sqrt{M+1}}$, the set $A_\lambda$ has empty interior. Applying proposition \ref{intprop} with $K = \{ \lambda \}$, $b = d/4$ we get $\epsilon, \rho > 0$.

Choose a disk $U \subseteq B(\lambda, \rho)$ centred at $\lambda$ small enough that for every $N$ large and any $z \in U$:
\[
|P'_N(z) - P'(\lambda)| < \frac{d}{2}.
\]
Integrating along the line segment from $\lambda_N$ to $\lambda$ (which is contained in $U$ since $\lambda_N$ and $\lambda$ are) as in the proof of proposition \ref{intprop} gives:
\[
|P_N(\lambda) - P'(\lambda)(\lambda - \lambda_N)| \leqslant \frac{d}{2}|\lambda - \lambda_N|.
\]
Then the reverse triangle inequality gives:
\[
\frac{d}{2}|\lambda_N - \lambda| \leqslant |P_N(\lambda)| \leqslant \frac{M}{1-|\lambda|} |\lambda|^{N+1}.
\]

Thus taking $C = \frac{2M}{d(1-|\lambda|)}$ gives the claimed inequality.

Choose $a > 0$ such that $3da < \epsilon$. Then for all $N$ large, we have:
\[
B(\lambda_N, a|\lambda|^{N+1}) \subseteq U.
\]
Moreover, for every $z$ in this ball we have:
\[
|z-\lambda| \leqslant (C+a)|\lambda|^{N+1},
\]
and hence:
\[
|z|^{N+1} \geqslant |\lambda|^{N+1}(1 - (C+a)|\lambda|^N)^{N+1}.
\]
We know that $(N+1)|\lambda|^N \rightarrow 0$ and so:
\[
|z|^{N+1} \geqslant \frac{1}{2}|\lambda|^{N+1}.
\]

Our choice of $U$ implies $|P'_N(z) - P'(\lambda)| < \frac{d}{2}$, therefore we have $\frac{d}{2} < |P'_N(z)| < \frac{3d}{2}$ for $z \in U$. Thus:
\[
|P_N(z)| \leqslant \frac{3d}{2} |z-\lambda_N| < \frac{3da}{2}|\lambda|^{N+1} < \frac{\epsilon}{2} |\lambda|^{N+1} \leqslant \epsilon |z|^{N+1}.
\]

Also, $|P'_N(z)| > \frac{d}{2} > b$ and since $z \in U$, $|z-\lambda| < \rho$. Thus proposition \ref{intprop} implies every point in the ball $B(\lambda_N, a|\lambda|^{N+1})$ is an interior point proving the second claim.

\end{proof}

\begin{proof}[Proof of Theorem \ref{thm3}]
Recall that any point $\lambda \in \partial \mathcal{M}_M$ satisfies $|\lambda| \leqslant \frac{1}{\sqrt{M+1}}$. Since $\mathbb{R}$ and the circle $|\lambda| = \frac{1}{\sqrt{M+1}}$ both have zero measure it suffices to consider $\lambda \in \partial \mathcal{M}_M \setminus \mathbb{R}$ with $|\lambda| < \frac{1}{\sqrt{M+1}}$.

By lemma \ref{m0lemma} we can find $a, C > 0$ and points $\lambda_N$ such that for all $N$ large:
\[
|\lambda_N - \lambda| \leqslant C|\lambda|^{N+1}, \quad B(\lambda_N, a|\lambda|^{N+1}) \subseteq \operatorname{int} \mathcal{M}_M.
\]

Let $r_N = (C+a)|\lambda|^{N+1}$. Then:
\[
B(\lambda_N, a|\lambda|^{N+1}) \subseteq B(\lambda, r_N).
\]
Notice also the smaller ball is disjoint from $\partial \mathcal{M}_M$. Comparing measures gives:
\begin{align*}
\frac{m(\partial \mathcal{M}_M \cap B(\lambda, r_N))}{m(B(\lambda, r_N))} &\leqslant 1 - \frac{\pi a^2|\lambda|^{2N+2}}{\pi(C+a)^2|\lambda|^{2N+2}} \\
&= 1 - \frac{a^2}{(C+a)^2} \\
&< 1.
\end{align*}

As $r_N \rightarrow 0$ this gives that $\lambda$ is not a point with Lebesgue density $1$ of $\partial \mathcal{M}_M$. Thus, by the Lebesgue density theorem, the collection of such $\lambda$ forms a set of measure zero.

\end{proof}

\end{document}